\documentclass[10pt]{amsart}
\usepackage{hyperref,mathrsfs,enumerate}
\usepackage[english]{babel}
\usepackage[all]{xy}

\newtheorem{theorem}{Theorem}

\newtheorem{lemma}{Lemma}[section]
\newtheorem{proposition}[lemma]{Proposition}

\theoremstyle{definition}
\newtheorem{definition}[lemma]{Definition}
\theoremstyle{remark}

\numberwithin{equation}{section}

\makeatletter
\newcommand{\labitem}[2]{\def\@itemlabel{#1} \item \def\@currentlabel{#1}\label{#2}}
\makeatother

\newcommand{\tr}[1]{\vphantom{#1}^t #1}

\begin{document}

\title{The $S$-adic Pisot conjecture on two letters}

\author[V.~Berth\'e]{Val\'erie Berth\'e}
\author[M.~Minervino]{Milton Minervino}
\author[W.~Steiner]{Wolfgang Steiner}
\address{LIAFA, CNRS UMR 7089, Universit\'e Paris Diderot -- Paris 7, Case 7014, 75205 Paris Cedex 13, FRANCE}
\email{berthe@liafa.univ-paris-diderot.fr, milton@liafa.univ-paris-diderot.fr, steiner@liafa.univ-paris-diderot.fr}

\author[J.~Thuswaldner]{J\"org Thuswaldner}
\address{Chair of Mathematics and Statistics, University of Leoben, A-8700 Leoben, AUSTRIA}
\email{joerg.thuswaldner@unileoben.ac.at}

\thanks{The authors are participants in the ANR/FWF project ``FAN -- Fractals and Numeration'' (ANR-12-IS01-0002, FWF grant I1136)}
\date{\today}

\begin{abstract}
We prove an extension of the well-known Pisot substitution conjecture to the $S$-adic symbolic setting on two letters. The proof relies on the use of Rauzy fractals  and on the fact that strong coincidences hold in this framework.
\end{abstract}

\maketitle

\section{Introduction}
A~symbolic substitution is a morphism of the free monoid. More generally, a substitution rule acts on a finite collection of tiles by first inflating them, and then  subdividing them into translates of tiles of the initial collection.
Substitutions thus generate symbolic dynamical system as well as tiling spaces.
The Pisot substitution conjecture states that any substitutive dynamical system has pure discrete spectrum, under the algebraic assumption that its expansion factor is a Pisot number (together with some extra assumption of irreducibility). 
Pure discrete spectrum means that the substitutive dynamical system is measurably conjugate  to a rotation on a compact abelian group. 
Pisot substitutions are thus expected to produce self-similar systems with long range order.
This conjecture has been proved in the two-letter case; see \cite{Barge-Diamond:02} together with \cite{Host} or~\cite{Hollander-Solomyak:03}.
For more on the Pisot substitution conjecture, see e.g.\ \cite{CANTBST,ABBLS}.

We prove an extension of the two-letter  Pisot substitution conjecture to the symbolic  $S$-adic framework, that is, for infinite words generated by iterating different substitutions in a prescribed order.
More precisely, an $S$-adic expansion of an infinite word~$\omega$ is given by a sequence of substitutions $\boldsymbol{\sigma} = (\sigma_n)_{n \in \mathbb{N}}$ (called directive sequence) and a sequence of letters $(i_n)_{n \in \mathbb{N}}$, such that $\omega = \lim_{n\to\infty} \sigma_0 \sigma_1 \cdots \sigma_n(i_{n+1})$. 
Such  expansions are widely studied. They occur e.g.\ under the term `mixed substitutions' or `multi-substitution' \cite{Gahler:2013,PV:2013}, or else as `fusion systems' \cite{PriebeFrank-Sadun11,PriebeFrank-Sadun14}.
They are closely connected to adic systems, such as considered e.g.\ in~\cite{Fisher:09}. Fore more on $S$-adic systems, see  \cite{Durand:00a,Durand:00b,Durand-Leroy-Richomme:13,Berthe-Delecroix}.

We consider the (unimodular) Pisot $S$-adic framework introduced in~\cite{Berthe-Steiner-Thuswaldner}, where unimodular means that the incidence matrices of the substitutions are unimodular. 
The $S$-adic Pisot condition is stated in terms of Lyapunov exponents: the second Lyapunov exponent associated with the shift space made of the directive sequences
$\boldsymbol{\sigma}$ and with the cocyles provided by the incidence matrices of the substitutions is negative. 
For a given directive sequence~$\boldsymbol{\sigma}$, let $\mathcal{L}_{\boldsymbol{\sigma}}^{(k)}$ stand for the language associated with the shifted directive sequence $(\sigma_{n+k})_{n \in \mathbb{N}}$. 
The Pisot condition implies in particular a uniform balancedness property for $\mathcal{L}_{\boldsymbol{\sigma}}^{(k)}$ (i.e., uniformly bounded symbolic discrepancy),  uniform with respect to some  infinite set of non-negative integers~$k$.  Note that the balancedness property   is proved in \cite{Sadun:15}     to characterize  topological  conjugacy  under changes in the lengths of the tiles  of the associated ${\mathbb R}$-action on  the  $1$-dimensional tilings.  

One difficulty when working in the $S$-adic framework is that no  natural candidate   exists for a left eigenvector (that is, for a stable space). 
Recall that the normalized left eigenvector (whose existence comes from the Perron-Frobenius Theorem) in the substitutive case provides in particular the measure of the tiles of the associated tiling of the line.
We introduce here a set of assumptions which, among other things, allows us to work with a generalized left eigenvector (that is, a stable space). 
We stress the fact that this vector is not canonically defined (contrarily to the right eigenvector). 
These assumptions will be implied by the $S$-adic Pisot condition.
We need first to guarantee the existence of a generalized right eigenvector~$\mathbf{u}$ (which yields the unstable space). 
The corresponding  conditions are natural and are stated in terms of primitivity of the directive  sequence~$\boldsymbol{\sigma}$ in the $S$-adic framework. 
We also require the directive sequences~$\boldsymbol{\sigma}$ to be recurrent (every finite combination of substitutions in the sequence occurs infinitely often), which yields unique ergodicity and strong convergence toward the generalized right eigendirection.
We then need rational independence of the coordinates of~$\mathbf{u}$. 
This is implied by the assumption of algebraically irreducibility, which states that the characteristic polynomial of  the incidence matrix of $\sigma_k \sigma_{k+1} \cdots \sigma_{\ell}$ is irreducible for all~$k$ and all sufficiently large~$\ell$.
Lastly, we then will be able to define a vector playing the role of a left eigenvector, by compactness together with primitivity and recurrence.

The starting point of the proofs in the two-letter substitutive case is that strong coincidences hold~\cite{Barge-Diamond:02}.
We also  prove  an analogous statement  as a starting point. However   the fact that strong  coincidences   hold at order $n$ (i.e.,
for the product $\sigma_0 \sigma_1 \cdots \sigma_n$)   does not necessarily imply strong coincidences at order $m$  for $m > n$ (which  holds  in the substitutive case)
makes the   extension to the present  framework  more delicate than it first  occurs. The proof heavily relies, among other things, on the recurrence  of the  directive sequence.
There are then two strategies, one based on  making explicit  the action of the rotation on the unit circle~\cite{Host}, whereas the approach of \cite{Hollander-Solomyak:03} uses the balanced pair algorithm. 
However, there seems to be no natural expression of the balanced pair algorithm in the $S$-adic framework.
Our strategy for proving discrete spectrum thus relies on the use of `Rauzy fractals' and follows~\cite{Host}. 
We recall that a Rauzy fractal is a set which is endowed with an exchange of pieces acting on it  that allows  to make explicit (by factorizing by a natural lattice) the maximal equicontinuous factor of the underlying symbolic dynamical system (see~\cite[Section~7.5.4]{Fog02}). For more on Rauzy fractals, see e.g. \cite{Fog02,ST09,CANTBST}.
More precisely, the strong coincidence condition implies that there is a well-defined exchange of pieces on the Rauzy fractal.
It remains to factorize this exchange of pieces in order to get a circle rotation. 
The factorization comes from \cite{Host} and extends directly from the substitutive case to the present $S$-adic framework.

\smallskip
Let us now state our results more precisely. 
(Although most of the terminology of the statements was defined briefly in the introduction above, we refer the reader to Section~\ref{sec:ingredients} for exact definitions.)
We just recall here that a sequence of substitutions $\boldsymbol{\sigma} = (\sigma_n)_{n\in\mathbb{N}}$ on the alphabet $\mathcal{A} = \{1,2\}$ satisfies the \emph{strong coincidence condition} if there is $n \in \mathbb{N}$ such that $\sigma_0 \sigma_1 \cdots \sigma_n(1)\in p_1 i \mathcal{A}^*$ and $\sigma_0 \sigma_1 \cdots \sigma_n(2)\in p_2 i \mathcal{A}^*$ for some $i\in\mathcal{A}$ and words $p_1,p_2\in\mathcal{A}^*$ with the same abelianization $\mathbf{l}(p_1) = \mathbf{l}(p_2)$.

\begin{theorem} \label{scc}
Let $\boldsymbol{\sigma} = (\sigma_n)_{n\in\mathbb{N}}$ be a primitive and algebraically irreducible sequence of substitutions over $\mathcal{A} = \{1,2\}$. Assume that there is $C > 0$ such that for each $\ell \in \mathbb{N}$, there is $n \ge 1$ with $(\sigma_{n},\ldots,\sigma_{n+\ell-1}) = (\sigma_{0},\ldots,\sigma_{\ell-1})$ and the language $\mathcal{L}_{\boldsymbol{\sigma}}^{(n+\ell)}$ is $C$-balanced. 
Then $\boldsymbol{\sigma}$ satisfies the strong coincidence condition.
\end{theorem}

Note that we do not assume unimodularity of the substitutions for strong coincidences, whereas we do for the following. 

\begin{theorem}\label{sadic2}
Let $\boldsymbol{\sigma} = (\sigma_n)_{n\in\mathbb{N}}$ be a primitive and algebraically irreducible  sequence of unimodular substitutions over~$\mathcal{A}=\{1,2\}$. Assume that there is $C > 0$ such that for each $\ell \in \mathbb{N}$, there is $n \ge 1$ with $(\sigma_{n},\ldots,\sigma_{n+\ell-1}) = (\sigma_{0},\ldots,\sigma_{\ell-1})$ and the language $\mathcal{L}_{\boldsymbol{\sigma}}^{(n+\ell)}$ is $C$-balanced. 
Then the $S$-adic shift $(X_{\boldsymbol{\sigma}},\Sigma,\mu)$, where $\Sigma$ stands for the shift and $\mu$ is the (unique) shift-invariant measure on $X$,  is measurably conjugate to a rotation on the circle~$\mathbb{S}^1$; in particular, it has pure discrete spectrum.
\end{theorem}

Let $S$ be a finite set of substitutions on $\mathcal{A}=\{1,2\}$ having invertible incidence matrices, and let $(D, \Sigma, \nu)$ with $D \subset S^\mathbb{N}$ be an (ergodic) shift equipped with a probability measure~$\nu$. 
With each $\boldsymbol{\sigma} = (\sigma_n)_{n\in\mathbb{N}} \in D$, associate the cocycle $A(\boldsymbol{\sigma}) = \tr{\!M}_0$ and denote the \emph{Lyapunov exponents} w.r.t.\ this cocycle by $\theta_1, \theta_2$. As in the more general situation of \cite[\S 6.3]{Berthe-Delecroix}, we say that $(D, \Sigma, \nu)$ satisfies the \emph{Pisot condition} if $\theta_1 > 0 > \theta_2$; see also \cite[Section 2.6]{Berthe-Steiner-Thuswaldner} for details.
With this notation, we are able to state the following theorem.

\begin{theorem}\label{th:ae}
Let $S$ be a finite set of unimodular substitutions on two letters, and let $(D, \Sigma, \nu)$ with $D \subset S^\mathbb{N}$ be a sofic shift that satisfies the Pisot condition. Assume that $\nu$ assigns positive measure to each (non-empty) cylinder, and that  there exists a cylinder corresponding to a substitution with positive incidence matrix.
Then, for $\nu$-almost all sequences $\boldsymbol{\sigma} \in D$ the $S$-adic shift $(X_{\boldsymbol{\sigma}},\Sigma,\mu)$ is measurably conjugate to a rotation on the circle~$\mathbb{S}^1$; in particular, it has pure discrete spectrum.
\end{theorem}

Let us describe briefly the organization of the paper. 
We recall the required definitions in Section~\ref{sec:ingredients}.
Section~\ref{sec:strongcoincidence} is devoted to the proof of the fact that the  strong coincidence condition holds (Theorem~\ref{scc}), and we conclude the proof of Theorems~\ref{sadic2} and~\ref{th:ae} in Section~\ref{sec:conjecture}.

\section{Ingredients}\label{sec:ingredients}

The following $S$-adic framework is defined in full detail and for a finite alphabet $\mathcal{A}=\{1,\ldots,d\}$ in \cite{Berthe-Steiner-Thuswaldner}.
We introduce the reader to some of the main notions and results for $d=2$.

\subsection{$S$-adic shifts}
Let $\mathcal{A} = \{1,2\}$ and let $\mathcal{A}^*$ denote the free monoid of finite words over~$\mathcal{A}$, endowed with
the concatenation of words as product operation.
A~\emph{substitution} $\sigma$ over~$\mathcal{A}$ is an endomorphism of~$\mathcal{A}^*$ sending non-empty words to non-empty words.
The \emph{incidence matrix} (or abelianization) of~$\sigma$ is the square matrix $M_\sigma = (|\sigma(j)|_i)_{i,j\in\mathcal{A}} \in \mathbb{N}^{2\times 2}$,
where  the notation $|w|_i$ stands for the number of occurrences of the letter~$i$ in $w \in \mathcal{A}^*$. 
The substitution $\sigma$ is said to be \emph{unimodular} if $|\!\det M_\sigma| = 1$.
The \emph{abelianization map} is defined by $\mathbf{l}:\ \mathcal{A}^* \to\mathbb{N}^2, \ w \mapsto {}^t(|w|_1,|w|_2)$.
Note that $M_\sigma\circ\mathbf{l} = \mathbf{l}\circ\sigma$.

Let $\boldsymbol{\sigma} = (\sigma_n)_{n\in\mathbb{N}}$ be a sequence of substitutions over the alphabet~$\mathcal{A}$.
For ease of notation we set $M_n = M_{\sigma_n}$ for $n \in \mathbb{N}$, and
\[ 
\sigma_{[k,\ell)} = \sigma_k \sigma_{k+1} \cdots \sigma_{\ell-1} \quad \mbox{and} \quad M_{[k,\ell)} = M_k M_{k+1} \cdots M_{\ell-1} \quad (0 \le k \le \ell).
\]
The sequence~$\boldsymbol{\sigma}$ is said to be \emph{primitive} if, for each $k \in \mathbb{N}$, $M_{[k,\ell)}$ is a positive matrix for some $\ell > k$.
We say that $\boldsymbol{\sigma}$ is \emph{algebraically irreducible} if, for each $k \in \mathbb{N}$, the characteristic polynomial of $M_{[k,\ell)}$ is irreducible for all sufficiently large~$\ell$.

Recall that $w\in\mathcal{A}^*$ is called a \emph{factor} of a finite or infinite word $v$ if it occurs at some position in $v$; it is a \emph{prefix} if it occurs at the beginning of $v$.
The \emph{language} associated with the sequence $(\sigma_{m+n})_{n\in\mathbb{N}}$ is 
\[
\mathcal{L}_{\boldsymbol{\sigma}}^{(m)} = \big\{w \in \mathcal{A}^*:\, \mbox{$w$ is a factor of $\sigma_{[m,n)}(i)$ for some $i \in\mathcal{A}$, $n\in\mathbb{N}$}\big\} \qquad (m \in \mathbb{N}).
\]
We say that a~pair of words $u, v \in \mathcal{A}^*$ with the same length is \emph{$C$-balanced} if 
\[
-C \le |u|_j - |v|_j \le C \quad \mbox{for all}\ j \in \mathcal{A}.
\]
A~language~$\mathcal{L}$ is $C$-\emph{balanced} if each pair of words $u, v \in \mathcal{L}$ with the same length is $C$-balanced.
It is said \emph{balanced} if it is $C$-balanced for some $C>0$.

The \emph{shift}~$\Sigma$ maps an infinite word $(\omega_n)_{n\in\mathbb{N}}$ to  $(\omega_{n+1})_{n\in\mathbb{N}}$.
A~dynamical system $(X,\Sigma)$ is a \emph{shift space} if $X$ is a closed shift-invariant set of infinite words over a finite alphabet, equipped with the product topology of the discrete topology. 

Given a sequence~$\boldsymbol{\sigma}$, let $S = \{\sigma_n:n\in\mathbb{N}\}$.
The \emph{$S$-adic shift} or \emph{$S$-adic system} with sequence~$\boldsymbol{\sigma}$ is the shift space $(X_{\boldsymbol{\sigma}}, \Sigma)$, where $X_{\boldsymbol{\sigma}}$ denotes the set of infinite words~$\omega$ such that each factor of~$\omega$ is an element of~$\mathcal{L}_{\boldsymbol{\sigma}}^{(0)}$.
If $\boldsymbol{\sigma}$ is primitive, then $X_{\boldsymbol{\sigma}}$ is the closure of the $\Sigma$-orbit of any limit word $\omega$ of~$\boldsymbol{\sigma}$, where $\omega\in\mathcal{A}^\mathbb{N}$ is a \emph{limit word} of~$\boldsymbol{\sigma}$ if there is a sequence of infinite words $(\omega^{(n)})_{n\in\mathbb{N}}$ with $\omega^{(0)} = \omega$ and $\omega^{(n)} = \sigma_n(\omega^{(n+1)})$ for all $n \in \mathbb{N}$.

Recall that a shift space $(X,\Sigma)$ is \emph{minimal} if every non-empty closed shift-invariant subset equals the whole set; it is called \emph{uniquely ergodic} if there exists a unique shift-invariant probability measure on~$X$.
Let $\mu$ be a shift-invariant measure defined on $(X,\Sigma)$. 
A~measurable eigenfunction of the system $(X,\Sigma,\mu)$ with associated eigenvalue $\alpha \in \mathbb{R}$ is an $L^2(X,\mu)$ function that satisfies $f(\Sigma^n(\omega)) = e^{2\pi i\alpha n} f(\omega)$ for all $n \in \mathbb{N}$ and $\omega \in X$. 
The system $(X,\Sigma, \mu)$ has \emph{pure discrete spectrum} if $L^2(X,\mu)$ is spanned by the measurable eigenfunctions. 
Furthermore, every dynamical system with pure discrete spectrum is  measurably  conjugate to a Kronecker system, i.e., a rotation on a compact abelian group; see~\cite{Walters:82}.

\subsection{Generalized Perron-Frobenius eigenvectors}
For a sequence of  non-nega\-tive matrices $(M_n)_{n\in\mathbb{N}}$, there exists by \cite[pp.~91--95]{Furstenberg:60} a positive vector $\mathbf{u} \in \mathbb{R}_+^2$ such that
\[
\bigcap_{n\in\mathbb{N}} M_{[0,n)}\, \mathbb{R}^2_+ = \mathbb{R}_+ \mathbf{u},
\]
provided there are indices $k_1 < \ell_1 \le k_2 < \ell_2 \le \cdots$ and a positive matrix~$B$ such that $B = M_{[k_1,\ell_1)} = M_{[k_2,\ell_2)} = \cdots$.
Thus, for primitive and recurrent sequences $\boldsymbol{\sigma}$, this vector exists and we call it the \emph{generalized right eigenvector of} $\boldsymbol{\sigma}$.
The following criterion for $\mathbf{u}$ to have rationally independent coordinates is \cite[Lemma~4.2]{Berthe-Steiner-Thuswaldner}.

\begin{lemma}\label{le:rational}
Let $\boldsymbol{\sigma}$ be an algebraically irreducible sequence of substitutions with generalized right eigenvector~$\mathbf{u}$ and balanced language~$\mathcal{L}_{\boldsymbol{\sigma}}$. Then the coordinates of~$\mathbf{u}$ are rationally independent.
\end{lemma}
 
Contrary to the cones $M_{[0,n)}\, \mathbb{R}_+^2$, there is no reason for the cones $\tr{(M_{[0,n)})}\, \mathbb{R}_+^2$ to be nested.
Therefore, the intersection of these cones does not define a generalized left eigenvector of~$\boldsymbol{\sigma}$. 
However, for a suitable choice of~$\mathbf{v}$, we have a subsequence $(n_k)_{k\in\mathbb{N}}$ such that the directions of $\mathbf{v}^{(n_k)} := \tr{(M_{[0,n_k)})}\, \mathbf{v}$ tend to that of~$\mathbf{v}$; in this case, $\mathbf{v}$ is called a \emph{recurrent left eigenvector}.

Under the assumptions of primitivity and recurrence of~$\boldsymbol{\sigma}$, given a strictly increasing sequence of non-negative integers~$(n_k)$, one can show that there is a recurrent left eigenvector $\mathbf{v} \in \mathbb{R}_{\ge0}^2 \setminus \{\mathbf{0}\}$ such that 
\begin{equation}\label{eq:recurrentcandidate}
\lim_{k\in K,\,k\to\infty} \frac{\mathbf{v}^{(n_k)}}{\|\mathbf{v}^{(n_k)}\|} = \lim_{k\in K,\,k\to\infty}\frac{{}^t(M_{[0,n_k)})\mathbf{v}}{\|{}^t(M_{[0,n_k)})\mathbf{v}\|}= \mathbf{v}
\end{equation} 
for some infinite set $K \subset \mathbb{N}$; see \cite[Lemma~5.7]{Berthe-Steiner-Thuswaldner}. Here and in the following, $\|\cdot\|$ denotes the maximum norm~$\|\cdot\|_\infty$. 
Note that the hypotheses of Lemma~\ref{le:rational} do not guarantee that the coordinates of~$\mathbf{v}$ are rationally independent.

We will work in the sequel with sequences~$\boldsymbol{\sigma}$ satisfying a list of conditions gathered in the following Property PRICE (which stands for Primitivity, Recurrence, algebraic Irreducibility, $C$-balancedness, and recurrent left Eigenvector). 
By \cite[Lemma~5.9]{Berthe-Steiner-Thuswaldner}, this property is a consequence of the assumptions of Theorem~\ref{scc}. 

\begin{definition}[Property PRICE]\label{def:star}
We say that a sequence $\boldsymbol{\sigma} = (\sigma_n)$ has Property \emph{PRICE} w.r.t.\ the strictly increasing sequences $(n_k)_{k\in\mathbb{N}}$ and $(\ell_k)_{k\in\mathbb{N}}$ and the vector $\mathbf{v} \in \mathbb{R}_{\ge0}^2 \setminus \{\mathbf{0}\}$ if the following conditions hold.
\begin{itemize}
\labitem{(P)}{defP}
There exists $h \in \mathbb{N}$ and a positive matrix~$B$ such that $M_{[\ell_k-h,\ell_k)} = B$ for all $k \in \mathbb{N}$.
\labitem{(R)}{defR}
We have $(\sigma_{n_k}, \sigma_{n_k+1}, \ldots,\sigma_{n_k+\ell_k-1}) = (\sigma_0, \sigma_1, \ldots,\sigma_{\ell_k-1})$ for all $k\in\mathbb{N}$.
\labitem{(I)}{defI}
The directive sequence~$\boldsymbol{\sigma}$ is algebraically irreducible.
\labitem{(C)}{defC}
There is $C > 0$ such that $\mathcal{L}_{\boldsymbol{\sigma}}^{(n_k+\ell_k)}$ is $C$-balanced for all $k\in\mathbb{N}$.
\labitem{(E)}{defE}
We have $\lim_{k\to\infty} \mathbf{v}^{(n_k)}/\|\mathbf{v}^{(n_k)}\|= \mathbf{v}$.
\end{itemize}
We also simply say that $\boldsymbol{\sigma}$ satisfies Property PRICE if the five conditions hold for some not explicitly specified strictly increasing sequences $(n_k)_{k\in\mathbb{N}}$ and $(\ell_k)_{k\in\mathbb{N}}$ and some $\mathbf{v} \in \mathbb{R}_{\ge0}^2 \setminus \{\mathbf{0}\}$.
\end{definition}

\subsection{Rauzy fractals}
For a vector $\mathbf{w} \in \mathbb{R}^2\setminus\{\mathbf{0}\}$, let 
\[  
\mathbf{w}^\perp = \{ \mathbf{x}\in\mathbb{R}^2:\, \langle \mathbf{x},\mathbf{w}\rangle = 0 \}
\]
be the line  orthogonal to~$\mathbf{w}$ containing the origin, equipped with the Lebesgue measure~$\lambda$. 
In particular, for $\mathbf{1} = \tr{(1,1)}$, $\mathbf{1}^\perp$ is the line  of vectors whose entries sum up to~$0$.
Let $\pi_{\mathbf{u},\mathbf{w}}$ be the projection along the direction~$\mathbf{u}$ onto~$\mathbf{w}^\perp$.

Given a primitive sequence of substitutions~$\boldsymbol{\sigma}$,  the \emph{Rauzy fractal} associated with~$\boldsymbol{\sigma}$ over~$\mathcal{A}$ is: 
\[ 
\mathcal{R} = \overline{\{\pi_{\mathbf{u},\mathbf{1}}\, \mathbf{l}(p):\, p \in \mathcal{A}^*,\ \mbox{$p$ is a prefix of a limit word of $\boldsymbol{\sigma}$}\}}. 
\]
The Rauzy fractal has natural \emph{refinements} defined by
\[ 
\mathcal{R}(w) = \overline{\{ \pi_{\mathbf{u},\mathbf{1}} \, \mathbf{l}(p):  p \in \mathcal{A}^*,\ \mbox{$pw$ is a prefix of a limit word of $\boldsymbol{\sigma}$}  \}} \quad (w\in\mathcal{A}^*).  
\]
If $w \in \mathcal{A}$, then $\mathcal{R}(w)$ is called a \emph{subtile}.
The  set  $\mathcal{R}$ is bounded if and only if $\mathcal{L}_{\boldsymbol{\sigma}}$ is balanced. 
If $\mathcal{L}_{\boldsymbol{\sigma}}$ is $C$-balanced, then $\mathcal{R} \subset [-C,C]^2 \cap \mathbf{1}^\perp$; see \cite[Lemma~4.1]{Berthe-Steiner-Thuswaldner}.
Note that  $\mathcal{R}$ is not necessarily an interval (however, it is an interval if the language $\mathcal{L}_{\boldsymbol{\sigma}}$ is Sturmian \cite{Fog02}.)

\subsection{Dynamical properties of $S$-adic shifts}\label{dynp}
For $\boldsymbol{\sigma}$ primitive, algebraically irreducible, and recurrent sequence of substitutions with balanced language~$\mathcal{L}_{\boldsymbol{\sigma}}$, the \emph{representation map} 
\[ 
\varphi:\, X_{\boldsymbol{\sigma}} \to \mathcal{R},\quad u_0u_1\cdots \mapsto \bigcap_{n\in\mathbb{N}} \mathcal{R}(u_0u_1\cdots u_n) 
\]
is well-defined, continuous and surjective; for more details, see \cite[Lemma~8.3]{Berthe-Steiner-Thuswaldner}.

Suppose that the strong coincidence condition holds. 
Then the \emph{exchange of pieces}
\[ 
E: \mathcal{R}\to\mathcal{R},\quad \mathbf{x}\mapsto\mathbf{x}+\pi_{\mathbf{u},\mathbf{1}}\,\mathbf{e}_i\quad\text{ if } \mathbf{x}\in\mathcal{R}(i),  
\]
is well-defined $\lambda$-almost everywhere on $\mathcal{R}$.

The following results appear in \cite[Theorem~1]{Berthe-Steiner-Thuswaldner}. 
The assumptions on the directive sequence~$\boldsymbol{\sigma}$ are the ones of  Theorem~\ref{sadic2}.

\begin{proposition}\label{sadicres}
Let $\boldsymbol{\sigma} = (\sigma_n)_{n\in\mathbb{N}}$ be a primitive and algebraically irreducible  sequence of unimodular substitutions over~$\mathcal{A}=\{1,2\}$. 
Assume that there is $C > 0$ such that for each $\ell \in \mathbb{N}$, there is $n \ge 1$ with $(\sigma_{n},\ldots,\sigma_{n+\ell-1}) = (\sigma_{0},\ldots,\sigma_{\ell-1})$ and the language $\mathcal{L}_{\boldsymbol{\sigma}}^{(n+\ell)}$ is $C$-balanced. 
Then the following results are true.
\begin{enumerate}
\item\label{min} 
The $S$-adic shift $(X_{\boldsymbol{\sigma}},\Sigma)$ is minimal and uniquely ergodic.  Let $\mu$ stand for its  unique invariant measure. 
\item 
Each subtile~$\mathcal{R}(i)$, $i \in \mathcal{A}$, of the Rauzy fractal~$\mathcal{R}$ is a compact set that is the closure of its interior; its boundary has zero Lebesgue measure~$\lambda$.
\item\label{1to1} 
If $\boldsymbol{\sigma}$ satisfies the strong coincidence condition, then the subtiles $\mathcal{R}(i)$, $i\in\mathcal{A}$, are mutually disjoint in measure, and
the $S$-adic shift $(X_{\boldsymbol{\sigma}},\Sigma,\mu)$ is measurably conjugate to the  exchange  of pieces $(\mathcal{R},E,\lambda)$ via~$\varphi$.
\end{enumerate}
\end{proposition}

We will consider in the sequel the one-dimensional lattice $\Lambda := \mathbf{1}^\perp \cap \mathbb{Z}^2 = \mathbb{Z}(\mathbf{e}_2 - \mathbf{e}_1)$. 
Let $\pi:\, \mathbf{1}^\perp \to \mathbf{1}^\perp/\Lambda$ be the canonical projection. 
Since $\pi_{\mathbf{u},\mathbf{1}}\,\mathbf{e}_2 \equiv \pi_{\mathbf{u},\mathbf{1}}\,\mathbf{e}_1 \bmod \Lambda$ holds, the canonical projection of~$E$ onto $\mathbf{1}^\perp/\Lambda \cong \mathbb{S}^1$
  is equal to the translation $\mathbf{x}\mapsto\mathbf{x} + \pi_{\mathbf{u},\mathbf{1}}\,\mathbf{e}_1$.
Then we have the following commutative diagram:
\begin{equation}\label{diag}
\begin{gathered}
\xymatrix{X_{\boldsymbol{\sigma}} \ar[r]^\varphi\ar[d]_\Sigma & \mathcal{R} \ar[r]^\pi\ar[d]_E & \mathbf{1}^\perp/\Lambda\ar[d]_{+\pi_{\mathbf{u},\mathbf{1}}\mathbf{e}_1} \\ X_{\boldsymbol{\sigma}} \ar[r]_\varphi & \mathcal{R}\ar[r]_\pi  &\mathbf{1}^\perp/\Lambda } 
\end{gathered}
\end{equation}

\section{Strong coincidence}\label{sec:strongcoincidence}
We recall the formalism of \emph{geometric substitution} introduced in \cite{Arnoux-Ito:01}. 
For $[\mathbf{x},i] \in \mathbb{Z}^2\times\mathcal{A}$, let
\[  
E_1(\sigma)[\mathbf{x},i] = \big\{ [M_\sigma \mathbf{x}+\mathbf{l}(p),j]:\, j\in\mathcal{A}, \, p\in\mathcal{A}^* \text{ and } pj \text{ is a prefix of }\sigma(i) \big\}. 
\]
Then strong coincidence holds (on two letters) if and only if there exists $n \in \mathbb{N}$ such that $E_1(\sigma_{[0,n)})[\mathbf{0},1] \cap  E_1(\sigma_{[0,n)})[\mathbf{0},2] \neq \emptyset$.

We identify each $[\mathbf{x},i] \in\mathbb{Z}^2\times\mathcal{A}$ with the segment $\mathbf{x} + [0,1)\, \mathbf{e}_i$. 
Define the \emph{height} (with respect to~$\mathbf{u}$ and~$\mathbf{v}$) of a point $\mathbf{x} = t \mathbf{u} + \pi_{\mathbf{u},\mathbf{v}}\, \mathbf{x} \in\mathbb{R}^2$ by $H(\mathbf{x}) := t \in \mathbb{R}$.

According to the terminology introduced in \cite{Barge-Diamond:02}, a~\emph{configuration} (of segments) $\mathcal{K}$ of size~$m$ with respect to a vector~$\mathbf{w}$ is a collection of $m$ distinct segments $[\mathbf{x},i] \in \mathbb{Z}^2 \times \mathcal{A}$ such that some translate of~$\mathbf{w}^\perp$ intersects the interior of each element of~$\mathcal{K}$ (the corresponding points thus have the same height with respect to  $\mathbf{u}$ and~$\mathbf{w}$).
The $n$-th iterate is 
\[
\mathcal{K}^{(n)} = \big\{ E_1(\sigma_{[0,n)})[\mathbf{x},i]:\, [\mathbf{x},i] \in \mathcal{K} \big\}.
\]
Note  that  he $n$-th iterate of a configuration is not a configuration  of segments but a union of ``broken lines''.

Observe that, by \cite[Proposition~4.3 and Lemma~4.1]{Berthe-Steiner-Thuswaldner} for a primitive, algebraically irreducible, and recurrent sequence of substitutions~$\boldsymbol{\sigma}$ with $C$-balanced language~$\mathcal{L}_{\boldsymbol{\sigma}}$, we have $\lim_{n\to\infty} \pi_{\mathbf{u},\mathbf{1}}\, M_{[0,n)} \mathbf{x} = \mathbf{0}$ for each $\mathbf{x} \in \mathbb{R}^2$ and $\| \pi_{\mathbf{u},\mathbf{1}}\,\mathbf{l}(p)\|\le C$ for all prefixes~$p$ of limit words of~$\boldsymbol{\sigma}$. 
Thus, for each sufficiently large~$n$, the vertices of~$\mathcal{K}^{(n)}$ are in 
\[
T_{\mathbf{u},C} := \{\mathbf{x} \in \mathbb{Z}^2:\, \|\pi_{\mathbf{u},\mathbf{1}}\, \mathbf{x}\| < C+1\},
\]
and we may consider only $\mathcal{K} \subset T_{\mathbf{u},C}$. (This corresponds to    \cite[Lemma 2]{Barge-Diamond:02} which is  stated in the substitutive case.)
In particular,  $ \{[\mathbf{0},1], [\mathbf{0},2] \}$  is a configuration as soon  as $\mathbf {w}$ has positive entries. Obviously, thus configuration it is contained in  $ T_{\mathbf{u},C}$.

We say that $\mathcal{K}$ has an $n$-\emph{coincidence} if there exist
$[\mathbf{x},i], [\mathbf{y},j] \in \mathcal{K}$ such that $E_1(\sigma_{[0,n)})[\mathbf{x},i]\cap E_1(\sigma_{[0,n)})[\mathbf{y},j]\neq\emptyset$.
Given a set $J \subseteq \mathbb{N}$, we say that a configuration~$\mathcal{K}$ is $J$-\emph{coincident} if $\mathcal{K}$ has an $n$-coincidence for some $n\in J$.
Observe that $n$-coincidence does not necessarily imply $m$-coincidence for $m > n$. 
However, translating all vertices of a configuration by a fixed vector does not change the property of being $J$-coincident. 

We first prove the following proposition, generalizing the proof of~\cite[Theorem 1]{Barge-Diamond:02}. 

\begin{proposition} \label{p:1}
Assume that the sequence of substitutions $\boldsymbol{\sigma} = (\sigma_n)_{n\in\mathbb{N}}$ over the alphabet $\mathcal{A} = \{1,2\}$ has Property PRICE w.r.t.\ the sequences $(n_k)_{k\in\mathbb{N}}$ and $(\ell_k)_{k\in\mathbb{N}}$ and the vector~$\mathbf{v}$, and that 
\begin{equation}
\mathbf{v}^\perp \cap \mathbb{Z}^2 \cap (T_{\mathbf{u},C} - T_{\mathbf{u},C}) = \{\mathbf{0}\}, \label{e:vrat}
\end{equation}
with $C$ such that $\mathcal{L}_{\boldsymbol{\sigma}}$ is $C$-balanced.
Then $\boldsymbol{\sigma}$ satisfies the strong coincidence condition.
\end{proposition}

Note that (\ref{e:vrat}) holds in particular when~$\mathbf{v}$ has rationally independent coordinates.

\begin{proof}
Let $(n_k)_{k\in\mathbb{N}}$, $(\ell_k)_{k\in\mathbb{N}}$ be the sequences of property PRICE.
Consider the sets 
\[ 
J_h = \big\{ n_{k_0} + n_{k_1} + \cdots + n_{k_s} : s \ge 0,\, n_{k_j} + \ell_{k_j} \leq \ell_{k_{j+1}} \ \forall\, 0 \le j < s,\, k_0 \ge h\big\}.  
\] 
Note that $k_{j } < k_{j+1}$, for $ 0 \le j < s$, since  $n_{k_j} + \ell_{k_j} \leq \ell_{k_{j+1}}$ implies in particular that  $\ell_{k_j} < \ell_{k_{j+1}}$.
Given such a sum in~$J_h$, repeatedly applying~\ref{defR} we get
\begin{equation}\label{rec}
\sigma_{[0,n_{k_0} + n_{k_1} + \cdots + n_{k_s})} = \sigma_{[0,n_{k_s})}\cdots \sigma_{[0,n_{k_1})} \sigma_{[0,n_{k_0})}.
\end{equation}  

We only consider in this proof  configurations with respect to~$\mathbf{v}$. Let $\mathcal{D}_h$ be the set of \emph{not} $J_h$-coincident configurations 
that are contained in~$T_{\mathbf{u},C}$, and $\mathcal{D} = \bigcup_{h\in\mathbb{N}} \mathcal{D}_h$. 
Since $J_0 \supset J_1 \supset \cdots$, we have $\mathcal{D}_0 \subseteq \mathcal{D}_1\subseteq \cdots$.
As $\mathbf{u} \in \mathbb{R}_+^2$ and $\mathbf{v} \in \mathbb{R}_{\ge0}^2 \setminus \{\mathbf{0}\}$ are not orthogonal (by Lemma~\ref{le:rational}), $\mathcal{D}$ contains only finitely many configurations up to translation.
Moreover, with each configuration $\mathcal{K}  \in \mathcal{D}_h$, all translates of~$\mathcal{K}$ that are in~$T_{\mathbf{u},C}$ are also contained in~$\mathcal{D}_h$, by the translation-invariance of $J_h$-coincidence. 
Therefore, we have $\mathcal{D}_h = \mathcal{D}$ for all sufficiently large~$h$.

Let $\mathcal{K}$ be a configuration in~$\mathcal{D}$ of maximal size.
There exists an interval~$I$ of positive length such that, for every $t \in I$, $\mathbf{v}^\perp + t \mathbf{u}$ intersects each of the segments of~$\mathcal{K}$ in its interior.  Indeed,    \eqref{e:vrat} implies that   $\mathbf{v} $ has  positive coordinates.
By property~\ref{defE}, the same holds for $(\mathbf{v}^{(n_k)})^\perp + t \mathbf{u}$, provided that $k$ is sufficiently large.

Consider now $\mathcal{K}^{(n_k)}$, with $k$ large enough such that all the following hold: $\mathcal{D}_k = \mathcal{D}$, all segments of~$\mathcal{K}$ intersect $(\mathbf{v}^{(n_k)})^\perp + t \mathbf{u}$ for all $t \in I$, and all vertices of $\mathcal{K}^{(n_k)}$ are in~$T_{\mathbf{u},C}$. 
Let $\{\mathbf{p}_1, \mathbf{p}_2, \ldots, \mathbf{p}_{r_k}\}$ be the set of vertices of~$\mathcal{K}^{(n_k)}$ that are contained in $M_{[0,n_k)} \big((\mathbf{v}^{(n_k)})^\perp + I\, \mathbf{u}\big) =  \mathbf{v}^\perp + I\, M_{[0,n_k)} \mathbf{u}$. 
By~\eqref{e:vrat}, no two points in $\mathbb{Z}^2 \cap T_{\mathbf{u},C}$ have the same height. Therefore, we can assume w.l.o.g.\ that $H(\mathbf{p}_1) < H(\mathbf{p}_2) < \cdots < H(\mathbf{p}_{r_k})$. 

 For all $1 \le j \le r_k$, the configuration 
\[ 
\mathcal{K}^{(n_k)}_j := \{ [\mathbf{x},i] \in\mathcal{K}^{(n_k)}:\, \big(\mathbf{x} + [0,1)\, \mathbf{e}_i\big) \cap \big(\mathbf{v}^\perp + H(\mathbf{p}_j)\, \mathbf{u}\big) \neq \emptyset \}  
\]
has the same size as~$\mathcal{K}$ because $\mathcal{K}$ is not $J_k$-coincident and, for each $[\mathbf{x}',i'] \in \mathcal{K}$, the collection of segments in $E_1(\sigma_{[0,n_k)})[\mathbf{x}',i']$ forms a broken line from $M_{[0,n_k)} \mathbf{x}'$ to $M_{[0,n_k)} (\mathbf{x}'+\mathbf{e}_{i'})$ that intersects, for each $t \in I$, the line $\mathbf{v}^\perp + t\, \mathbf{u}$ exactly once. (Observe that the stripe $\mathbf{v}^\perp + I\, M_{[0,n_k)} \mathbf{u}$ has two complementary components, each of which contains one of the endpoints $M_{[0,n_k)} \mathbf{x}'$ and $M_{[0,n_k)} (\mathbf{x}'+\mathbf{e}_{i'})$ because $\mathbf{x}'$ and $\mathbf{x}'+\mathbf{e}_{i'}$ lie in different complementary components of the stripe $(\mathbf{v}^{(n_k)})^\perp + I \mathbf{u}$.)

Moreover, we have $\mathcal{K}^{(n_k)}_j \in \mathcal{D}$.
Indeed, take $h \ge k$ such that $\ell_h \ge n_k + \ell_k$.
If $\mathcal{K}^{(n_k)}_j$ were not in $\mathcal{D} = \mathcal{D}_h$, then $\mathcal{K}^{(n_k)}_j$ would have an $m$-coincidence for some $m \in J_h$.
Write $m= n_{k_0} + n_{k_1} + \cdots + n_{k_s}$.  One has $m+n_k \in J_k$, since $ n_k  +\ell_k \leq  \ell_h \leq  \ell_{k_0} $.
But then $\mathcal{K}$ would have an ($m+n_k$)-coincidence because 
\[ 
E_1(\sigma_{[0,m)}) \mathcal{K}^{(n_k)} = E_1(\sigma_{[0,m)}) E_1(\sigma_{[0,n_k)}) \mathcal{K} = E_1(\sigma_{[0,m+n_k)})\mathcal{K},
\] 
contradicting that $\mathcal{K} \in \mathcal{D} = \mathcal{D}_k$. 

For all $2 \le j \le r_k$, the configurations $\mathcal{K}^{(n_k)}_{j-1}$ and $\mathcal{K}^{(n_k)}_j$ differ only by segments ending and beginning at~$\mathbf{p}_j$, and the number of segments in $\mathcal{K}^{(n_k)}$ ending and beginning at~$\mathbf{p}_j$ is thus the same. 
Let $\mathcal{K}'$ be equal to $\mathcal{K}^{(n_k)}_{j-1}$, with the segments ending at~$\mathbf{p}_j$ removed.
We have the following possibilities.
\begin{enumerate}
\item\label{caseA}
Two segments of~$\mathcal{K}^{(n_k)}_{j-1}$ end at~$\mathbf{p}_j$. 
Then $\mathcal{K}^{(n_k)}_j = \mathcal{K}' \cup \{[\mathbf{p}_j,1], [\mathbf{p}_j,2]\}$. 
\item\label{caseB}
One segment of~$\mathcal{K}^{(n_k)}_{j-1}$ ends at~$\mathbf{p}_j$, and either $\mathcal{K}' \cup \{[\mathbf{p}_j,1]\}$ or $\mathcal{K}' \cup \{[\mathbf{p}_j,2]\}$ is in~$\mathcal{D}$.
Then $\mathcal{K}^{(n_k)}_j = \mathcal{K}' \cup \{[\mathbf{p}_j,i]\}$ with $i$ such that $\mathcal{K}' \cup \{[\mathbf{p}_j,i]\} \in \mathcal{D}$.
\item
One segment of~$\mathcal{K}^{(n_k)}_{j-1}$ ends at~$\mathbf{p}_j$, and both $\mathcal{K}' \cup \{[\mathbf{p}_j,1]\}$, $\mathcal{K}' \cup \{[\mathbf{p}_j,2]\}$ are in~$\mathcal{D}$. 
Since the size of~$\mathcal{K}$ is maximal in~$\mathcal{D}$, we have $\mathcal{K}' \cup \{[\mathbf{p}_j,1], [\mathbf{p}_j,2]\} \notin \mathcal{D}$, hence $E_1(\sigma_{[0,n)}) [\mathbf{p}_j,1] \cap E_1(\sigma_{[0,n)}) [\mathbf{p}_j,2] \neq \emptyset$ for some $n \in \mathbb{N}$. 
Then also $E_1(\sigma_{[0,n)}) [\mathbf{0},1] \cap E_1(\sigma_{[0,n)}) [\mathbf{0},2] \neq \emptyset$, thus the strong coincidence condition holds.
\end{enumerate}

Assume now that the strong coincidence condition does not hold. 
Hence we are always either in case \eqref{caseA} or case~\eqref{caseB}.
Then $\mathcal{D}$ and the relative positions of the segments within~$\mathcal{K}^{(n_k)}_{j-1}$ entirely determine~$\mathcal{K}^{(n_k)}_j$.
(Note that $\mathbf{p}_j$ is the endpoint of the segments in $\mathcal{K}^{(n_k)}_{j-1}$, disregarding~$\mathbf{p}_{j-1}$, with minimal height.)

Recall that $\mathcal{D}$ contains up to translation only finitely many configurations, and denote the number of such configurations by~$c$.
Then we have $\mathcal{K}^{(n_k)}_{a+b} = \mathcal{K}^{(n_k)}_a + \mathbf{t}_k$ for some $1 \le a,b \le c$, and some translation vector~$\mathbf{t}_k = \mathbf{p}_{b+a} - \mathbf{p}_a \in\mathbb{Z}^2$ (provided that $r_k \ge 2c$). 
Consequently, we have $\mathcal{K}^{(n_k)}_{j+b} = \mathcal{K}^{(n_k)}_j + \mathbf{t}_k$ for all $a \le j \le r_k-b$, and thus $\mathcal{K}^{(n_k)}_{a+\ell b} = \mathcal{K}^{(n_k)}_a + \ell \,\mathbf{t}_k$, for all $\ell$ such that $a + \ell b \leq r_k$.

Let now $k \to \infty$. 
Then the stripe $\mathbf{v}^\perp + I\, M_{[0,n_k)} \mathbf{u}$ becomes wider and wider as~$k$ grows, hence $r_k \to \infty$.
Since there are only finitely many possibilites for~$\mathbf{t}_k$, there exists thus a $\mathbf{t} \in \mathbb{Z}^2$ such that arbitrarily large multiples of~$\mathbf{t}$ are translation vectors of configurations in~$\mathcal{D}$. 
Since all configurations in~$\mathcal{D}$ are in~$T_{\mathbf{u},C}$, the vector~$\mathbf{t}$ must be a scalar multiple of~$\mathbf{u}$, in contradiction with Lemma~\ref{le:rational}.
This proves that the strong coincidence condition holds.
\end{proof}

To prove Theorem~\ref{scc}, we show that the conditions of Proposition~\ref{p:1} are fulfilled for some shifted sequence $(\sigma_{n+h})_{n\in\mathbb{N}}$. 

\begin{proof}[Proof of Theorem~\ref{scc}]
Let $\boldsymbol{\sigma}$ satisfy the assumptions of Theorem~\ref{scc}.
By \cite[Lemma~5.9]{Berthe-Steiner-Thuswaldner}, property PRICE holds for some sequences $(n_k)$, $(\ell_k)$, a vector~$\mathbf{v}\in\mathbb{R}^2_{\ge 0}$ and a balancedness constant~$C$. 
If $\mathbf{v}$ has rationally independent coordinates, then we can apply Proposition~\ref{p:1} directly. To cover the contrary case, assume in the following that $\mathbf{v}$ is a multiple of a rational vector. 

Consider $\mathbf{x} \in \mathbb{Z}^2 \setminus \{\mathbf{0}\}$.
We have
\begin{align*}
\langle \mathbf{x},\mathbf{v}^{(h)} \rangle & = \langle \mathbf{x}, \tr{(M_{[0,h)})} \mathbf{v} \rangle = \langle M_{[0,h)} \mathbf{x}, \mathbf{v} \rangle = \langle H(M_{[0,h)} \mathbf{x}) \mathbf{u} + \pi_{\mathbf{u},\mathbf{v}}\, M_{[0,h)} \mathbf{x}, \mathbf{v} \rangle \\
& = H(M_{[0,h)} \mathbf{x})\, \langle \mathbf{u}, \mathbf{v} \rangle.
\end{align*}
As $\mathbf{u}$ has rationally independent coordinates by Lemma~\ref{le:rational} and $\mathbf{v}$ is a multiple of a rational vector, we have $\langle \mathbf{u}, \mathbf{v} \rangle \neq 0$.
We have $\lim_{h\to\infty} \pi_{\mathbf{u},\mathbf{v}}\, M_{[0,h)} \mathbf{x} = \mathbf{0}$ by \cite[Proposition~4.3]{Berthe-Steiner-Thuswaldner} and $M_{[0,h)} \mathbf{x} \in \mathbb{Z}^2 \setminus \{\mathbf{0}\}$ since $\boldsymbol{\sigma}$ is algebraically irreducible, thus $H(M_{[0,h)} \mathbf{x}) \ne 0$ for all sufficiently large~$h$.
We conclude that, for each $\mathbf{x} \in \mathbb{Z}^2 \setminus \{\mathbf{0}\}$, there is $h_0(\mathbf{x})$ such that 
\begin{equation}\label{eq:hx}
\mathbf{x} \not\in (\mathbf{v}^{(h)})^\perp \hbox{ for all } h \ge h_0(\mathbf{x}).
\end{equation}
As $\mathbf{u} \in \mathbb{R}_+^2$, the set $\bigcup_{h\in \mathbb{N}} (\mathbf{v}^{(h)})^\perp \cap (T_{\mathbf{u},C} - T_{\mathbf{u},C})$ is bounded and, hence, its intersection with $\mathbb{Z}^2$ is finite. Thus \eqref{eq:hx} implies that $(\mathbf{v}^{(h)})^\perp \cap (T_{\mathbf{u},C} - T_{\mathbf{u},C}) \cap \mathbb{Z}^2 = \{\mathbf{0}\}$ for all sufficiently large~$h$. 

Choose now $h = n_k + \ell_k$ sufficiently large.
Then the language $\mathcal{L}_{\boldsymbol{\sigma}}^{(h)}$ is $C$-balanced by Property~\ref{defC}.
By \cite[Lemma~5.10]{Berthe-Steiner-Thuswaldner}, there exists $k_0$ such that the shifted sequence $(\sigma_{n+h})_{n\in\mathbb{N}}$ has property PRICE w.r.t.~$(n_{k+k_0})$ and $(\ell_{k+k_0}-h)$ and the vector $\mathbf{v}^{(h)}$.
Thus $(\sigma_{n+h})_{n\in\mathbb{N}}$ satisfies the strong coincidence condition by Proposition~\ref{p:1}, i.e., there exists $n$ such that $E_1(\sigma_{[h,n+h)}) [\mathbf{0},1] \cap E_1(\sigma_{[h,n+h)}) [\mathbf{0},2] \neq \emptyset$. 
This implies that $E_1(\sigma_{[0,n+h)}) [\mathbf{0},1] \cap E_1(\sigma_{[0,n+h)}) [\mathbf{0},2] \neq \emptyset$, which concludes the proof of the theorem. 
\end{proof}

\section{Proof of the $S$-adic Pisot conjecture}\label{sec:conjecture}
In this section we will deduce Theorem~\ref{sadic2} from Theorem~\ref{scc} and the following lemma, which is a generalization of a result of \cite{Host}; see also \cite[Section~6.3.3]{Queffelec:10}. 
Recall Section~\ref{dynp}, in particular the diagram~\eqref{diag}.
Theorem~\ref{th:ae} follows immediately from Theorem~\ref{sadic2}.

\begin{lemma}\label{host}
Let $\boldsymbol{\sigma}$ be as in Theorem~\ref{sadic2}. 
Then the map $\overline{\varphi} = \pi \circ \varphi$ is one-to-one $\mu$-almost everywhere on $X_{\boldsymbol{\sigma}}$.
\end{lemma}

\begin{proof}
Note first that it  is sufficient to show that whenever $\overline{\varphi}(u) = \overline{\varphi}(v)$, one can find a non-negative integer $n$ such that $\varphi(\Sigma^n u) = \varphi(\Sigma^n v)$.
The $\mu$-almost everywhere injectivity of~$\overline{\varphi}$ will thus come from the $\mu$-almost everywhere injectivity of~$\varphi$, which holds according to Proposition~\ref{sadicres}~\eqref{1to1}.

Let now $u,v \in X_{\boldsymbol{\sigma}}$ be  such that $\overline{\varphi}(u) = \overline{\varphi}(v)$. Consider the set 
\[ 
\big\{n \in \mathbb{N}:\, \mathbf{z}_n := \varphi(\Sigma^n u) - \varphi(\Sigma^n v) = 0\big\}. 
\]
By induction, one has $\mathbf{z}_n \in \Lambda = \mathbb{Z}\, (\mathbf{e}_2-\mathbf{e}_1)$ for all~$n$. 
Indeed, $\mathbf{z}_0 \in \Lambda$ because $\overline{\varphi}(u) - \overline{\varphi}(v) = 0$. 
Using $E\circ\varphi = \varphi\circ\Sigma$ we see, for all~$n$, that 
\[
\mathbf{z}_{n+1} - \mathbf{z}_n = \big(E^{n+1}(\varphi(u)) - E^n(\varphi(u))\big) - \big(E^{n+1}(\varphi(v)) - E^n(\varphi(v))\big) \in \big\{\mathbf{0},\pm (\mathbf{e}_2-\mathbf{e}_1)\big\}.
\]

Let $\pi_0 : \mathbf{1}^\perp \to \mathbb{R}$, $(x,-x)\mapsto x$ be the projection on the first coordinate.
Then $z_n := \pi_0(\mathbf{z}_n) \in\mathbb{Z}$ and $z_{n+1}-z_n \in \{0,\pm 1 \}$ for all $n$.
Since $\mathcal{L}_{\boldsymbol{\sigma}}$ is $C$-balanced, we know that $\varphi(X_{\boldsymbol{\sigma}}) = \mathcal{R} \subset [-C,C]^2 \cap \mathbf{1}^\perp$, thus $\pi_0(\varphi(\Sigma^n u)) \in [-C,C]$ for any $u\in X_{\boldsymbol{\sigma}}$.
Minimality of $(X_{\boldsymbol{\sigma}},\Sigma)$, given by Proposition~\ref{sadicres}~\eqref{min}, implies that the sets 
\begin{align*}
A &= \{ n \in \mathbb{N} : \pi_0(\varphi(\Sigma^n u)) > C-1 \}, \\
B &= \{ n \in \mathbb{N} : \pi_0(\varphi(\Sigma^n u)) < -C + 1\}
\end{align*} 
are relatively dense.
If $n\in A$, then $\pi_0(\varphi(\Sigma^n v))\leq C < \pi_0(\varphi(\Sigma^n u)) + 1$ and $z_n\geq 0$. Analogously, we deduce that, if $n\in B$, then $z_n\leq 0$. The result follows observing that, given $p\in A$, $q\in B$ such that $p\leq q$, we can find~$n$ such that $p\leq n \leq q$ and  $z_n = 0$, using the fact that $z_{n+1} - z_n \in \{0,\pm 1\}$ for all $n$.
\end{proof}

\begin{proof}[Proof of Theorem~\ref{sadic2}]
By Theorem~\ref{scc}, the strong coincidence condition holds, which implies that $(X_{\boldsymbol{\sigma}},\Sigma,\mu)$ is measurably conjugate to $(\mathcal{R},E,\lambda)$ via~$\varphi$ by Proposition~\ref{sadicres}~\eqref{1to1}. But by Lemma~\ref{host} we even have that $(X_{\boldsymbol{\sigma}},\Sigma,\mu)$ is measurably conjugate to $(\mathbf{1}^\perp/\Lambda,+\pi_{\mathbf{u},\mathbf{1}}\,\mathbf{e}_1,\lambda)$ via~$\overline{\varphi}$, with $\mathbf{1}^\perp/\Lambda \cong \mathbb{S}^1$.
\end{proof}

\begin{proof}[Proof of Theorem~\ref{th:ae}]
In view of Theorem~\ref{sadic2}, this is an immediate consequence of \cite[Theorem~2]{Berthe-Steiner-Thuswaldner}.
\end{proof}

\bibliographystyle{amsalpha}
\bibliography{sadic2}
\end{document}